\documentclass[a4paper, 10pt, oneside]{article}

\usepackage{tgpagella}
\usepackage{latexsym}
\usepackage[utf8]{inputenc}
\usepackage[T1]{fontenc}
\usepackage{amssymb}
\usepackage{amsfonts}
\usepackage{amsmath}
\usepackage{amsthm}
\usepackage{mathrsfs}
\usepackage{amscd}
\usepackage{bussproofs}
\usepackage{tikz-cd}
\usepackage[all]{xy}
\usepackage{comment}
\usepackage{enumitem}
\usepackage{relsize}
\usepackage{color}

\newcommand{\set}[2]{ \lbrace #1 \ \mid \ #2 \rbrace}
\newcommand{\qcr}[1]{\ulcorner #1 \urcorner}
\newcommand{\PA}{\textnormal{PA}}

\newcommand{\Th}{\textnormal{Th}}
\newcommand{\was}[1]{\{#1\}}
\newcommand{\CT}{\textnormal{CT}}

\newcommand{\KF}{\textnormal{KF}}
\newcommand{\form}{\textnormal{Form}}
\newcommand{\TB}{\textnormal{TB}}
\newcommand{\PT}{\textnormal{PT}}
\newcommand{\UTB}{\textnormal{UTB}}
\newcommand{\pt}{\textnormal{PT}}

\newcommand{\WKF}{\textnormal{WKF}}
\newcommand{\WPT}{\textnormal{WPT}}
\newcommand{\tot}{\textnormal{tot}}
\newcommand{\PTtot}{\PT^- + \intot}
\newcommand{\PTint}{\PT^- + \inter}
\newcommand{\val}[1]{{#1}^{\circ}}
\newcommand{\ConPA}{\textnormal{Con}_{\PA}}

\newcommand{\inter}{\textnormal{INT}}

\newcommand{\num}[1]{\underline{#1}}
\newcommand{\Sent}{\textnormal{Sent}}
\newcommand{\intot}{\textnormal{INT}\upharpoonright_{\tot}}

\newcommand{\df}[1]{\textbf{#1}}
\newcommand{\CC}{\textnormal{CC}}
\newcommand{\GC}{\textnormal{GC}}
\newcommand{\dpt}{\textnormal{dp}}
\newcommand{\ind}{\textnormal{ind}}
\newcommand{\Asn}{\textnormal{Asn}}
\newcommand{\ucl}{\textnormal{ucl}}
\newcommand{\term}{\textnormal{Term}}
\newcommand{\PAT}{\textnormal{PAT}}
\newcommand{\FV}{\textnormal{FV}}
\newcommand{\Safe}{\textnormal{Safe}}
\newcommand{\Lang}{\mathcal{L}}
\newcommand{\M}{\mathcal{M}}

\newcommand{\poprawka}[1]{#1}
\newcommand{\usuniete}[1]{}
\newcommand{\dousuniecia}[1]{}

\usepackage{hyperref}

\usepackage{pgf}
\usepackage{xxcolor} 

\newtheorem{tw}{Theorem}
\newtheorem{fakt}[tw]{Fact}
\newtheorem{stwierdzenie}[tw]{Proposition}
\newtheorem{wniosek}[tw]{Corollary}
\newtheorem{lemat}[tw]{Lemma}

\theoremstyle{definition}
\newtheorem{definicja}[tw]{Definition}
\newtheorem{uwaga}[tw]{Remark}

\newtheorem{konwencja}{Convention}

\begin{document}
	\title{Models \poprawka{of} Positive Truth}
	\author{Mateusz Łełyk, Bartosz Wcisło}
	
	\maketitle
	
\begin{abstract}
	This paper is a follow-up to \cite{CWL}. We give a strenghtening of the main result on the semantical non-conservativity of the theory of $\pt^-$ with internal induction for total formulae ($\pt^- + \inter(\tot)$). We show that if to $\pt^-$ the axiom of internal induction for all arithmetical formulae is added ($\pt^- + \inter$), then this theory is semantically stronger than $\pt^- + \inter(\tot)$. In particular the latter is not relatively truth definable (in the sense of \cite{fujimoto}) in the former. Last but not least we provide an axiomatic theory of truth which meets the requirements put forward by Fischer and Horsten in \cite{fischerhorsten}.
\end{abstract}

\section{Introduction}

\subsection{Axiomatic Theories of Truth}

\emph{Axiomatic theories of truth} is a branch of mathematical logic and philosophy which studies the properties of formal theories generated in the following way:

\begin{enumerate}
	\item We take a \emph{base theory} $B$ which we demand to be sufficiently strong to (strongly) represent basic syntactical operations.
	\item We extend the language of $B$ by adding one new unary predicate $T$ and some axioms for it so that the resulting theory $\Th$ prove all sentences of the form
	\[T(\qcr{\phi})\equiv \phi\]
	for $\phi$ \emph{in the language of our base theory $B$}.
\end{enumerate} 
For a brief introduction to the subject see \cite{stanford} and for a more complete one---\cite{halbach}. The big question that this paper answers in a tiny part is the following: how various axioms for the truth predicate influence its \emph{strength}. For the purpose of investigating this question we focus on the truth theories with \emph{Peano Arithmetic} as a base theory. The notion of \textit{strength} may enjoy many different explications. For example, the simplest one is given by inclusion of sets of consequences: we might say that $\Th_1$ is not weaker than $\Th_2$ if and only if $\Th_1$ proves all the axioms of $\Th_2$. For many applications this is too fine-grained: many theories, intuitively differing in strength, become incomparable out of not-that-important reasons (obviously this is not a formal notion). Better adjusted notion was introduced by Kentaro Fujimoto in \cite{fujimoto} and is a special kind of interpretability. We recall the definition:

\begin{definicja}
	\begin{enumerate}
		\item Let $\Th_1$ and $\Th_2$ be axiomatic truth theories and let $T_{\Theta_2}$ be the truth predicate of $\Th_2$. For any sentence $\Theta$ of $\mathcal{L}_{\Th_2}$ and a formula $\phi(x)\in\mathcal{L}_{\Th_1}$ with precisely one free variable let
		\[\Theta^{\phi(x)}\]
		denote the $\mathcal{L}_{\Th_1}$ sentence which results from $\Theta$ by substituting $\phi(x)$ for every occurrence $T_{\Th_2}(x)$ (and renaming variables, if necessary).
		\item We say that $\Th_1$ \emph{relatively truth defines} $\Th_2$ if and only if there exists a formula $\phi(x)\in\mathcal{L}_{\Th_1}$ such that for any axiom $\Theta$ of $\Th_2$
		\[\Th_1\vdash \Theta^{\phi(x)}\]
		If $\Th_2$ relatively truth defines $\Th_1$ we will denote it by $\Th_1\leq_F\Th_2$\footnote{\poprawka{''$F$'' stands for ''Fujimoto.''}}.
	\end{enumerate}
\end{definicja}

	In terms of interpretations, relative truth definability is a $\mathcal{L} _{\PA}$-conservative interpretation between truth theories (for the terminology related to interpretations see e.g. \cite{visser}). It was argued in \cite{fujimoto} that relative truth definability provides a good explication of \textit{epistemological reduction} between truth theories. We may treat it as an explication of a notion of strength: $\Th_1$ is \emph{Fujimoto-stronger} than $\Th_2$ if and only if $\Th_1$ relatively truth defines $\Th_2$ but not \textit{vice versa}. This relation will be denoted by $\lneq_F$. 
	
\subsection{Strength relative to $\PA$}

In some philosophical debates, especially the ones related to the deflationism, the need for a differently oriented formal explication of strength seems to emerge. It has been claimed (most importantly in \cite{shapiro}, \cite{ketland} and \cite{horsten}) that deflationary thesis that truth is a "simple" (aka "light", "metaphysically thin") notion implies that the deflationary theory of truth should be \emph{conservative} over $\PA$.\footnote{This thesis however has been recently criticised at length in \cite{ciesinn}} Let us recall that a theory of truth can be conservative over $\PA$ in two senses:

\begin{definicja}
	Let $\Th$ be a theory of truth. 
	\begin{enumerate}
		\item We say that $\Th$ is \emph{proof-theoretically conservative} over $\PA$ if and only if for every $\phi\in\mathcal{L}_{\PA}$, if $\Th\vdash \phi$, then $\PA\vdash \phi$.\footnote{This property is also called \textit{syntactical conservativity}}.
		\item We say that $\Th$ is \textit{model-theoretically conservative} over $\PA$ if and only if every model $\mathcal{M}$ of $\PA$ admits an \emph{expansion} to a model of $\Th$.\footnote{This relation is also called \textit{semantical conservativity}}.
	\end{enumerate}
\end{definicja}

\begin{uwaga}
	Note that, in the definition of model-theoretical conservativity, we do not merely demand every model to have an \emph{extension} to a model of $\Th$, in which case both notion of conservativity would be the same. We say that $\mathcal{M}'$ is an \emph{expansion} of $\mathcal{M}$ if \poprawka{$\M$ and $\M '$ are the same model, except that $\M '$ carries interpretation of additional function, relation and constant symbols.} 
\end{uwaga}

The two notions lead in the natural way to the following generalisations:

\begin{definicja}
	Let $\Th_1$ and $\Th_2$ be two truth theories.
	\begin{enumerate}
		\item We say that $\Th_1$ is \textit{proof-theoretically not stronger} than $\Th_2$ if every $\mathcal{L}_{\PA}$ sentence provable in $\Th_1$ is provable in $\Th_2$. If $\Th_1$ is proof-theoretically not stronger than $\Th_1$, we will denote it with $\Th_1\leq_{P}\Th_2$.\footnote{$P$ is meant to abbreviate "Proof".}
		\item We say that $\Th_1$ is \textit{model-theoretically not stronger} than $\Th_2$ if every model which can be expanded to a model of $\Th_2$, can be expanded to a model of $\Th_1$. If $\Th_1$ is syntactically not stronger than $\Th_1$, we will denote it with $\Th_1\leq_{M}\Th_2$.\footnote{$M$ is meant to abbreviate "Model".}
	\end{enumerate}
\end{definicja}

Obviously we say that $\Th_2$ is proof-theoretically (model-theoretically) stronger than $\Th_1$ if $\Th_1\leq_P \Th_2$ but $\Th_2\nleq_P\Th_1$ (respectively, $\Th_1~\leq_M~\Th_2$ but $\Th_2~\nleq_M~\Th_1$). This relation will be denoted $\lneq_P$ ($\lneq_M$ respectively). 

Let us observe that the three notions of strength introduced above can be ordered with respect to their "granularity". Indeed, for any theories $\Th_1$ and $\Th_2$ we have:
\begin{equation}\nonumber\label{strength1}\tag{$FMP$}
\Th_1\leq_F \Th_2 \implies \Th_1\leq_M\Th_2 \implies  \Th_1\leq_P\Th_2.
\end{equation}
Hence also
\begin{equation}\nonumber\label{strength2}\tag{$\neg PMF$}
\Th_2\nleq_P \Th_1 \implies \Th_2\nleq_M\Th_1 \implies  \Th_2\nleq_P\Th_2
\end{equation}
Having three different notions of strength makes it possible to decide not only whether one theory of truth is stronger than another one, but also \emph{how much} stronger it is.

\subsection{Compositional Positive Truth and its Extensions}
Before continuing let us introduce some handy notational conventions:
\begin{konwencja}\text{}
	\begin{enumerate}
	\item By using variables $\phi, \psi$ we implicitly restrict quantification to (G\"odel codes of) arithmetical sentences. For example, by $\forall \phi \ \ \Psi( \phi)$ we mean $\forall x \ \ \bigl(\Sent(x) \rightarrow \Psi(x)\bigr)$ and by $\exists \phi \ \ \Psi(\phi)$ we mean $\exists x \ \ \bigl(\Sent(x) \wedge \Psi(x)\bigr).$ For brevity, we will sometimes also use variables $\phi, \psi$ to run over arithmetical formulae, whenever it is clear from the context which one we mean; similarly
	\begin{enumerate}
	\item $\phi(v), \psi(v)$ run over arithmetical formulae with at most one indicated free variable (i.e. $\phi(v)$ is either a formula with exactly one free variable or a sentence); $\phi(\bar{x})$, $\psi(\bar{x})\ldots$ run over arbitrary arithmetical formulae. 
	\item  $s,t$ run over codes of closed arithmetical terms;
	\item $v,v_1,v_2, \ldots, w, w_1, w_2, \ldots$ run over codes of variables;
	\end{enumerate}
	\item $\form_{\Lang_{\PA}}(x)$, $\form^{\leq 1}_{\Lang_{\PA}}(x)$, $\Sent_{\Lang_{\PA}}(x)$ are natural arithmetical formulae strongly representing in $\PA$ the sets of (G\"odel codes of) formulae of $\Lang_{\PA}$, formulae  of $\Lang_{\PA}$ with at most one free variable, sentences of $\Lang_{\PA}$, respectively.
	\item if $\phi$ is a $\Lang_{\PA}$ formula, then $\qcr{\phi}$ denotes either its G\"odel code or the numeral denoting the G\"odel code of $\phi$ (context-dependently). This is the unique way of using $\qcr{\cdot}$ in this paper.
	\item  to enhance readability we suppress the formulae representing the syntactic operations. For example we write $\Phi(\psi \wedge \eta)$ instead of $\Phi(x) \wedge$ "$x$ is the conjunction of $\psi$ and $\eta$", similarly, we write $\Phi(\psi(t))$ instead of $\Phi(x) \wedge x = \textnormal{Subst}(\psi,t)$;
	\item $\num{x}$ denotes the (G\"odel code of) standard numeral for $x$, i.e. $\qcr{\underbrace{S\ldots S(0)}_{x \textnormal{ times }S}}$
	\item $\val{y}$ is the standard arithmetically definable function representing the value of term (coded by) $y$.
	\end{enumerate}
\end{konwencja}

The main objective of this study is to measure the strength of theories that are compositional, but do not enjoy the global axiom for commutativity with the negation, i.e.
\begin{equation}\tag{NEG}\label{equat_neg}
\forall \phi \ \ \bigl(T(\neg\phi)\equiv \neg T(\phi)\bigr)
\end{equation}
Let us formulate the theories which will be of the main interest.

\begin{definicja}
	$\pt^-$ is the axiomatic truth theory with the following axioms for the truth predicate:
	\dousuniecia{Hej, to majo być te kornery, czy nie majo? Usunąłem, ale nie jestem pewien czy słusznie.}
	\begin{enumerate}
		\item 
		\begin{enumerate}
			\item $\forall s,t \ \ \bigl(T(s=t) \equiv (\val{s}=\val{t})\bigr)$
			\item  $\forall s,t \ \ \bigl(T(\neg s=t) \equiv (\val{s}\neq\val{t})\bigr)$
		\end{enumerate}

		\item \begin{enumerate}
			\item $\forall \phi,\psi \ \ \bigl(T(\phi\vee \psi)\equiv T(\phi)\vee T(\psi)\bigr)$
			\item $\forall \phi,\psi \ \ \bigl(T(\neg (\phi\vee \psi))\equiv T(\neg\phi)\wedge T(\neg\psi)\bigr)$
		\end{enumerate}
		\item \begin{enumerate}
			\item $\forall v \forall \phi(v)\ \ \bigl(T(\exists v\phi)\equiv \exists x\ \ T(\phi(\num{x}))\bigr)$
			\item $\forall v\forall \phi(v)\ \ \bigl(T(\neg\exists v\phi)\equiv \forall x\ \ T(\neg\phi(\num{x}))\bigr)$
		\end{enumerate}
		
		\item $\forall \phi\ \ \bigl(T(\neg\neg\phi)\equiv T(\phi)\bigr)$
		\item $\forall \phi\forall s,t\Bigl(\val{s}=\val{t}\rightarrow \Big(T(\phi(s))\equiv T(\phi(t))\Big)\Bigr)$
	\end{enumerate}
\end{definicja}

In the arithmetized language, we treat $\wedge$ and $\forall$ as symbols defined contextually, i.e. $\phi\wedge \psi = \neg(\neg\phi\vee\neg\psi)$ and $\forall v\phi = \neg\exists v\neg\phi$. Then it is easy to check that the following sentences are provable in $\pt^-$:
\begin{enumerate}
	\item $\forall \phi,\psi \left(T(\phi\wedge\psi)\equiv \left(T(\phi)\wedge T(\psi)\right)\right)$.
	\item $\forall \phi,\psi \left(T\left(\neg\left(\phi\wedge\psi\right)\right)\equiv \left(T\left(\neg\phi\right)\vee T\left(\qcr{\neg\psi}\right)\right)\right)$.
	\item $\forall v\forall \phi(v) \left(T(\forall v\phi)\equiv \left(\forall x\ \ T(\phi(\num{x})\right)\right)$
	\item $\forall v\forall \phi(v) \left(T(\neg\forall v\phi)\equiv \left(\exists x\ \ T(\neg\phi(\num{x})\right)\right)$
\end{enumerate} 

In $\pt^-$ the internal logic (i.e. the logic of all \emph{true} sentences) is modelled after the Strong Kleene Scheme. Let us observe that axioms of $\pt^-$ make it possible to accept a disjunction $\phi\vee \psi$ as true simply on the basis of the truth of one of $\phi$ and $\psi$ and regardless of whether the second one has its truth value determined. The second theory we will study is more cautious in this respect. Let us define 
\[\tot(\phi(v)) := \form^{\leq 1}(\phi(v))\wedge \forall x \bigl( T(\phi(\num{x}))\vee T(\neg\phi(\num{x}))\bigr)\]
In particular if $\psi$ is a sentence, then 
$$\PAT^-\vdash\tot(\psi) \equiv \biggl(T(\psi)\vee T(\neg\psi)\biggr).$$
where $\PAT^-$ is the extension of $\PA$ in $\mathcal{L}_{\PA}\cup \was{T}$, with no non-logical axioms for $T$.
\begin{definicja}
	$\WPT^-$ is the axiomatic truth theory with the following axioms for the truth predicate:
	\begin{enumerate}
		\item 
		\begin{enumerate}
			\item $\forall s,t \ \ \bigl(T(s=t) \equiv (\val{s}=\val{t})\bigr)$
			\item  $\forall s,t \ \ \bigl(T(\neg s=t) \equiv (\val{s}\neq\val{t})\bigr)$
		\end{enumerate}
		
		\item \begin{enumerate}
			\item $\forall \phi,\psi \ \ \bigl(T(\phi\vee \psi)\equiv \bigl(\tot(\phi)\wedge \tot(\psi) \wedge \bigl(T(\phi)\vee T(\psi)\bigr)\bigr)$
			\item $\forall \phi,\psi \ \ \bigl(T(\neg (\phi\vee \psi))\equiv T(\neg\phi)\wedge T(\neg\psi)\bigr)$
		\end{enumerate}
		\item \begin{enumerate}
			\item $\forall v \forall \phi(v)\ \ \bigl(T(\exists v\phi)\equiv \tot(\phi(v))\wedge \exists x T(\phi(\num{x}))\bigr)$
			\item $\forall v\forall \phi(v)\ \ \bigl(T(\neg\exists v\phi)\equiv \forall x\ \ T(\neg\phi(\num{x}))\bigr)$
		\end{enumerate}
		\item $\forall \phi\ \ \bigl(T(\neg\neg\phi)\equiv T(\phi)\bigr)$
		\item $\forall \phi\forall s,t\bigl(\val{s}=\val{t}\rightarrow T(\phi(s))\equiv T(\phi(t))\bigr)$
	\end{enumerate}
\end{definicja}

Using the above mentioned conventions regarding $\wedge$ and $\forall$, it is an easy exercise to show that the following sentences are provable in $\WPT^-$:
\begin{enumerate}
		\item $\forall \phi,\psi \left(T(\phi\wedge\psi)\equiv \left(T(\phi)\wedge T(\psi)\right)\right)$.
		\item $\forall \phi,\psi \left(T\left(\neg\left(\phi\wedge\psi\right)\right)\equiv \left(\tot(\phi)\wedge\tot(\psi) \wedge \left(T\left(\neg\phi\right)\vee T\left(\neg\psi\right)\right)\right)\right)$.
		\item $\forall v\forall \phi(v) \left(T(\forall v\phi)\equiv \left(\forall x\ \ T(\phi(\num{x})\right)\right)$
		\item $\forall v\forall \phi(v) \left(T(\neg\forall v\phi)\equiv \left(\tot(\phi(v))\wedge\exists x\ \ T(\neg\phi(\num{x}))\right)\right)$
\end{enumerate}
In $\WPT^-$ the internal logic is modelled after the \emph{Weak Kleene Scheme}. $(\textnormal{W})\pt^-$ can be seen as a natural stratified counterpart of $(\textnormal{W})\KF^-$\footnote{For the definition of all mentioned theories \poprawka{not defined in this paper,} consult \cite{halbach} or \cite{fujimoto} (for $\WKF$).} Since in particular $(\textnormal{W})\PT^-$ is a subtheory of $(\textnormal{W})\KF^-$ and the latter are well known to be model-theoretically conservative over $\PA$ (see \cite{cant}; we will outline direct proof of model-theoretical conservativity of $\pt^-$ in Section \ref{sect_int}), we have
\begin{stwierdzenie}
	$\pt^-$ and $\WPT^-$ are model-theoretically conservative.
\end{stwierdzenie}
In particular we see that the axiom \eqref{equat_neg} may contribute to the strength of truth theories: it is easy to see that $(W)\pt^- + \eqref{equat_neg}$ is deductively equivalent to the theory $\CT^-$, hence in particular by the well-known theorem of Lachlan (see \cite{halbach},\cite{kaye}) it is not semantically conservative.

For the sake of convenience let us isolate one easily noticeable feature of $\pt^-$ and $\WPT^-$:

\begin{definicja}[$\UTB$]
	Let $\phi(x_0,\ldots,x_n)$ be any arithmetical formula. $\UTB^-(\phi)$ is the following $\Lang_T$ sentence
	\begin{equation}\label{UTB}\tag{$\UTB^-(\phi)$}
	\forall t_0\ldots t_n\ \  \bigl(T(\qcr{\phi(t_0,\ldots, t_n)})\equiv \phi(\val{t_0},\ldots,\val{t_n})\bigr)
	\end{equation}
	Define 
	\[\UTB^- := \set{\UTB^-(\phi(x_0,\ldots,x_n))}{\phi(x_0,\ldots,x_n)\in\Lang_{\PA}}\]
	And $\UTB$ to be the extensions of $\UTB^-$ with all instantiations of induction scheme with $\Lang_T$ formulae.
\end{definicja}

\begin{fakt}\label{utb_w_pt_wpt}
	Both $\PT^-$ and $\WPT^-$ prove $\UTB^-$.
\end{fakt}

In \cite{fischer2},\cite{fischer3} and \cite{fischerhorsten} (this last philosophical motivation was summarized also in \cite{CWL}) authors motivated the need for a weak theory of truth which would be able to prove in a single sentence the fact that every arithmetical formula satisfy the induction scheme. Such a fact can be naturally expressed by an $\mathcal{L}_T$ sentence
\begin{equation}\label{equat_ind}\tag{\textnormal{INT}}
\forall \phi(x)\ \ \biggl[\bigl(\forall x \bigl(T(\phi(\num{x}))\rightarrow T(\phi(\num{x+1})\bigr)\bigr)\longrightarrow \bigl(T(\phi(0))\rightarrow \forall x T(\phi(\num{x}))\bigr)\biggr]
\end{equation}
For further usage let us abbreviate the formula
\[\bigl(\forall x \bigl(T(\phi(\num{x}))\rightarrow T(\phi(\num{x+1})\bigr)\bigr)\longrightarrow \bigl(T(\phi(0))\rightarrow \forall x T(\phi(\num{x}))\bigr)\]
by $\inter(\phi(x))$. Using Fact \ref{utb_w_pt_wpt}, we see that both $\PT^- + \inter$ and $\WPT^- + \inter$ can prove any arithmetical instance of the induction schema in a uniform way, for each formula using the same finitely many axioms\footnote{The proof is really easy: we fix $\phi(x)$ (with parameters), prove the instantiation of the \ref{UTB} scheme for $\phi$ and substitute $\phi(x)$ for $T(\phi(\num{x}))$ in \ref{equat_ind}}. In particular, it can be finitely axiomatised by taking $I\Sigma_1$ together with axioms for the truth predicate from $\pt^-$ and $\eqref{equat_ind}$. To achieve this goal, however, none of the discussed theories uses the full strength of \eqref{equat_ind}.
By \ref{UTB} every standard formula is total, provably in $\WPT^-$. Hence it makes good sense to consider a version of \eqref{equat_ind} restricted to total formulae, i.e.

\begin{equation}\label{equat_int_tot}\tag{$\intot$}
\forall \phi(v)\ \ \biggl[\tot(\phi(v))\longrightarrow \inter(\phi(v))\biggr]
\end{equation}

The theory $\pt^- + \intot$ was claimed to be model-theoretically conservative in \cite{fischer} (and then used in \cite{fischer2},\cite{fischer3} and \cite{fischerhorsten} as such). However, as shown in \cite{CWL}, the proof of its conservativity contained an essential gap and no prime model of $\PA$\footnote{For the definition of all notions from the model theory of $\PA$ see \cite{kaye}.} admits an expansion to the model of $\PTtot$. Moreover, it was shown that every recursively saturated model of $\PA$ can be expanded to a model of this theory. In particular, $\pt^-+\intot$ is model-theoretically stronger than $\pt^-$ and weaker than $\UTB$ and $\CT^-$. 

In the current study, we further approximate the class of models expandable to $\pt^-+\intot$ and compare the strength of $\UTB$ with the strength of $\pt^- + \inter$. Moreover we show that $\WPT^- + \inter$ is model theoretically conservative and meets the requirements posed in \cite{fischerhorsten}. Our results jointly with some well-known facts from the literature give the following picture of interdependencies between proof-theoretically conservative theories of truth:

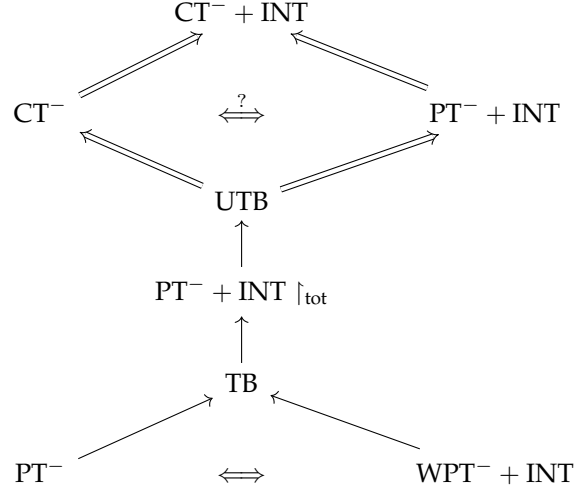
\begin{center}
\begin{tikzcd}
	 & \CT^- + \inter &\\
\CT^-\ar[ur, Rightarrow]&      \overset{?}{\Longleftrightarrow}     & \pt^-+\inter\ar[ul, Rightarrow]\\
 & \UTB\ar[ur, Rightarrow]\ar[ul, Rightarrow] &\\
 & \PTtot\ar[u] &\\
 & \TB\ar[u]&\\
\pt^-\ar[ur]& \Longleftrightarrow & \WPT^- +  \inter\ar[ul]
\end{tikzcd} 
\end{center}
\bigskip
\dousuniecia{Ale profeska.}
where $\longrightarrow$ stands for $\lneq_M$ and $\Longrightarrow$ for $\leq_M$. The question whether any of $\Longrightarrow$ arrows is in fact a $\longrightarrow$ arrow is open. Similarly, the relation between classes of models of $\CT^-$ and $\PT^- + \inter$ is unknown.

\section{Models of $\pt^- + \intot$}\label{sect_tot}

In the paper \cite{CWL}, it has been shown that $\PTtot$ is not semantically conservative over $\PA$ and, moreover, any (not necessarily countable) recursively saturated model of $\PA$ admits an expansion to a model of $\PTtot$. The nonconservativity result has been obtained by demonstrating that no prime model of $\PA$ can be expanded in such a way. Now, we will show a strengthening of that result. Let us first recall one definition.

\begin{definicja}
Let $M$ be a model of $\PA$. We say that $M$ is \df{short recursively saturated} if any recursive type (with finitely many parametres from $M$) of the form $p(x) = \set{x<a \wedge \phi_i(x)}{i \in \mathbb{N}}$ is realised in $M$ where $a$ is some fixed parameter from $M$.
\end{definicja}


In other words, a model is short recursively saturated if it realises all types which are finitely realised \emph{below} some fixed element. This notion is strictly weaker than full recursive saturation. For example, the standard model $\mathbb{N}$ is short recursively saturated but not recursively saturated. More generally, a countable model is short recursively saturated if and only if it has a recursively saturated elementary end extension, see \cite{smorynski}, Theorem 2.8.


\begin{tw}\label{tw_srecsat_to_petetot}
 Let $M \models \PA$ and suppose that $M$ has an expansion $(M,T) \models \PTtot$. Then $M$ is short recursively saturated.
\end{tw}

The proof of our theorem will closely parallel the proof of Theorem 4 from \cite{CWL}. In particular, we will again use a propositional construction invented by Smith.

\begin{definicja} \label{df_alternatives_stopping}
 Let $(\alpha_i)_{i\leq c}, (\beta_i)_{i \leq c}$ be any sequences of sentences. We define the \df{alternative with stopping condition} $(\alpha_i)$ 
 \begin{displaymath}
  \bigvee_{i=i_0}^{c,\alpha} \beta_i
 \end{displaymath}
 by backwards induction on $i_0$ as follows:
 \begin{enumerate}
  \item $\bigvee_{i=c}^{c,\alpha} \beta_i = \alpha_c \wedge \beta_c.$
  \item $\bigvee_{i=k}^{c,\alpha} \beta_i = \neg (\alpha_k \wedge \neg \beta_k) \wedge \Bigl( (\alpha_k \wedge \beta_k) \vee \big(\neg \alpha_k \wedge \bigvee_{i=k+1}^{c,\alpha} \beta_i \big) \Bigr).$
 \end{enumerate}
\end{definicja}
We may think that this is a formalisation in propositional logic of the following instruction: for $i$ from $i_0$ up to $c$, search for the first number $j$ such that $\alpha_j$ holds and then check whether also $\beta_j$ holds. \emph{Then stop your search}. The whole formula is true if this $\beta_j$ is true and is false if either $\beta_j$ is false or there is no $j$ such that $\alpha_j$ holds. It turns out that this intuition may be partially recovered in theories of truth, even if one does not assume that the truth predicate satisfies induction axioms.

\begin{lemat} \label{lem_total_stopping_disjunction}
 Fix $(M,T) \models \PT^-.$ Suppose that $(\alpha_i)_{i \leq c}$ is a sequence of arithmetical sentences coded in $M$. Suppose that the least $j$ such that $T(\alpha_j)$ holds is standard, say $j=j_0$, and that for any $k \leq  j_0$ either $T\left( \beta_k\right)$ or $T \left(\neg \beta_k\right)$ holds. Then 
 \begin{enumerate}
  \item $(M,T) \models T \Bigl(\bigvee_{i=c}^{c,\alpha} \beta_i \Bigr) \equiv T \left(\beta_{j_0}\right).$
  \item $(M,T) \models T \Bigl(\neg \bigvee_{i=c}^{c,\alpha} \beta_i \Bigr) \equiv T \left( \neg \beta_{j_0}\right).$
 \end{enumerate}

\end{lemat}

For a proof, see \cite{CWL}, Lemma 2.3.
Now we are ready to prove that any model of $\PA$ expandable to a model of $\PTtot$ is short recursively saturated.
\begin{proof}
 Fix any recursive type $p(x) = (\phi_i(x) \wedge x<a)_{i \in \omega}$ (with a parameter $a$) and suppose that for any finite set $\phi_0, \ldots, \phi_k$ there is some $b_k<a$ such that $M \models \phi_0(b_k) \wedge \ldots \wedge \phi_k(b_k).$ Let 
 \begin{eqnarray*}
 \alpha_0(x) & = & \neg x < \num{a} \vee \neg \phi_0(x) \\
  \alpha_{j+1}(x) & = & x<\num{a} \wedge \phi_0(x) \wedge \ldots \wedge \phi_{j}(x) \wedge \neg \phi_{j+1}(x).
 \end{eqnarray*}
 In a sense, formulae $\alpha_j(x)$ measure how much of the type $p$ is realised by $x$. Now, if the type $p$ is ommitted in the model $M$, then for any $x$, there exists a standard $j$ such that $(M,T) \models T\alpha_j(\num{x}).$ 
 Let $\beta_j(y)$ be defined as 
 \begin{displaymath}
  \beta_j(y) = y< \num{a} \wedge \phi_0(y) \wedge \ldots \wedge \phi_{j+1}(y). 
 \end{displaymath}
 Now, fix any nonstandard $c$ and consider the (nonstandard) formula
 \begin{displaymath}
  \phi(x,y) = \bigvee_{i=0}^{c,\alpha(x)} \beta_i(y).
 \end{displaymath}
 By Lemma \ref{lem_total_stopping_disjunction} and our assumption that the type $p$ is omitted in $M$, the sentence $\phi(\num{x},\num{y})$ is either true or false for any fixed $x,y \in M.$ But this means that the formula $\phi(x,y)$ is total. One can check that then a formula 
 \begin{displaymath}
  \psi(z) = \exists y \forall x<z \phi(x,y)
 \end{displaymath}
 is also total. Note that this formula intuitively says that there is a $y$ which satisfies more of a type $p$ than any of the elements of $M$ up to $z.$ Now, we will show that $\psi(z)$ is progressive, i.e.,
 \begin{displaymath}
  (M,T) \models \forall z \Big(T \psi(\num{z}) \rightarrow T \psi(\num{z+1})\Big).
 \end{displaymath}
 Fix any $z$ and suppose that $T\psi(\num{z})$ holds. Then there exists a $y$ such that $T (\qcr{\forall x<z \ \phi(x,\num{y})}).$ Now, let $j$ be the least number such that  $T \alpha_j(\num{z}).$ Since $j$ is a standard number and $p$ is a type, there exists $y'$ such that $\phi_0(y') \wedge \ldots \wedge \phi_{j+1}(y')$ holds in $M$, i.e. $(M,T) \models T \beta_j(\num{y'})$. Let $y'' = y$ if also 
\begin{displaymath}
 T \Big(\phi_0(\num{y}) \wedge \ldots \wedge \phi_{j+1}(\num{y}) \Big)
\end{displaymath}
and $y''=y'$ otherwise. In other hand, we fix either $y$ or $y'$, whichever satisfies ''more'' formulae $\phi_i.$ One readily checks that then
\begin{displaymath}
 (M,T) \models \forall x < z+1 \ \phi(\num{x},\num{y''}).
\end{displaymath}

We have shown that the formula $\psi(z)$ is total and progressive. By the internal induction for total formulae this means that 
\begin{displaymath}
 (M,T) \models \forall z \ T \psi(\num{z}).
\end{displaymath}
In particular, we have $T \psi(\num{a}),$ where $a$ is the parameter used as a bound in the type $p$. But then for some $d$, we have 
\begin{displaymath}
 (M,T) \models \forall x<a \ T\phi(\num{x}, \num{d}).
\end{displaymath}
Now, since $p$ is a type, for an arbitrary $k \in \omega$, there exists some $x<d$ such that $\phi_0(x) \wedge \ldots \wedge \phi_k(x)$. Since $(M,T) \models T \phi(\num{x},\num{d}),$ it follows that $d<a \wedge \phi_0(d) \wedge \ldots \wedge \phi_{k+1}(d).$ As we have chosen an arbitrary $k$, we see  that actually $d$ satisfies the type $p$. We conclude that $M$ is short recursively saturated.
\end{proof}

Let us summarize our findings from \cite{CWL} and this paper:
\begin{itemize}
 \item Any recursively saturated model of $\PA$ (possibly uncountable) admits an expansion to a model of $\PTtot.$
 \item If a model $M$ expands to a model of $\PTtot$, then it is short recursively saturated.
\end{itemize}
Unfortunately, we do not know whether any of the implications reverses. 

\poprawka{Cieśliński and Engstr\"om have (independetly) found the following characterisation the class of models of $\PA$ which admit an expansion to a model of $\TB$, i.e., the truth theory axiomatised with the induction scheme for the whole language and the following scheme of Tarski's biconditionals:
\begin{displaymath}
T (\qcr{\phi}) \equiv \phi,
\end{displaymath}
where $\phi$ is an arithmetical sentence. 
\begin{tw}[Cieśliński, Engstr\"om]\footnote{See \cite{cies_disquotational}, Theorem 7.}
	Let $M$ be a nonstandard model of $\PA$. Then the following are equivalent:
	\begin{enumerate}
		\item $M$ admits an expansion to a model $(M,T) \models \TB$.
		\item There exists an element $c \in M$ such that for all (standard) arithmetical sentences $\phi$, $M \models \qcr{\phi} \in c$ iff $M \models \phi$, i.e., $M$ codes its own theory.
	\end{enumerate}
\end{tw}
It can be easily shown that every nonstandard short recursively saturated model $M \models \PA$ satisfies the second item of the above characterisation. Hence, every short recursively saturated model of $\PA$ admits an expansion to a model of $\TB.$ Thus we obtain the following corollary:
\begin{wniosek} \label{cor_modele_pttot_sa_modelami_tb}
	$\TB \leq_M \PTtot$, i.e., every model $M \models \PA$ which admits an expansion to a model of $\PTtot$, also admits an expansion to a model of $\TB$.
\end{wniosek}}

\subsection{A non-result}

We are going to show that the method used to prove that every recursively saturated model of $\PA$ admits an expansion to a model of $\PT^-+\intot$ cannot be used to obtain a stricter upper bound on the class of models expandable to this theory (if such a stricter upper bound exists). Let us recall that in $\cite{CWL}$, it was shown that every recursively saturated model of $\PA$ can be expanded to a model of \poprawka{$\PT^- + \intot$}.  Let us recall the standard definition of a function which generates possible extensions for the $\PT^-$ truth predicate. 

\begin{konwencja}
	If $\M\models \PA$, then $\Sent_{\M}$, $\form^{\leq 1}_{\M}$, $\term_{\M}$ denote the set of sentences of $\Lang_{\PA}$, the set of formulae with at most one free variable and the set of terms, respectively, \emph{in the sense of $\M$.}
\end{konwencja}

\begin{definicja}\label{defin_oper_theta}
	Let $\mathcal{M}\models \PA$ and $A\subseteq \M$. Define:
	\begin{eqnarray}
	\Theta_{\mathcal{M}}(\phi,A)&:= &\ \ \mathcal{M}\models \exists s,t\ \ [\phi = (s=t)\wedge \val{s} = \val{t}]\nonumber\\
	&\vee&\ \  \mathcal{M}\models\exists s,t\ \ [\phi = \neg(s=t) \wedge  \val{s}\neq \val{t}]\nonumber\\
	&\vee&\ \ \exists \psi\in \Sent_{\mathcal{M}} [\mathcal{M}\models \phi = \neg\neg\psi]\wedge \psi\in A\nonumber\\
	&\vee&\ \ \exists \psi_1,\psi_2 \in \Sent_{\mathcal{M}}[\mathcal{M}\models\phi = (\psi_1\vee\psi_2)] \wedge \bigl(\psi_1\in A)\vee (\psi_2\in A)\bigr)\nonumber\\
	&\vee&\ \ \exists \psi_1,\psi_2 \in \Sent_{\mathcal{M}}[\mathcal{M}\models\phi = \neg(\psi_1\vee\psi_2)] \wedge \bigl( (\neg\psi_1\in A)\wedge (\neg\psi_2\in A)\bigr)\nonumber\\
	&\vee&\ \ \exists \psi(x)\in {\form}^1_{\mathcal{M}}[\mathcal{M}\models\phi = \exists x\psi] \wedge \exists x\in \M\ \ (\psi(\num{x})\in A)\nonumber\\
	&\vee&\ \ \exists \psi(x)\in {\form}^1_{\mathcal{M}}[\mathcal{M}\models\phi = \neg\exists x\psi] \wedge \forall x\in \M\ \ (\neg\psi(\num{x})\in A)\nonumber
	\end{eqnarray}
	Let $\Gamma^{\mathcal{M}}: \mathcal{P}(M)\rightarrow\mathcal{P}(M)$ be the function defined:
	\begin{equation}\label{Gamma}\tag{$\Gamma$}
	\Gamma^{\mathcal{M}}(A) = \set{\phi\in\mathcal{M}}{\Theta_{\mathcal{M}}(\phi,A)}.
	\end{equation}
	Let us now define:
	\begin{eqnarray}
	\Gamma^{\mathcal{M}}_0 &=& \Gamma^{\mathcal{M}}(\emptyset)\nonumber\\
	\Gamma^{\mathcal{M}}_{\alpha+1} &=& \Gamma^{\mathcal{M}}(\Gamma^{\mathcal{M}}_{\alpha})\nonumber\\
	\Gamma^{\mathcal{M}}_{\beta} &=& \bigcup_{\alpha<\beta}\Gamma^{\mathcal{M}}_{\alpha},\textnormal{ for $\beta$ a limit ordinal.}\nonumber
	\end{eqnarray}
	It can be checked that for some ordinal $\alpha$ we must get $\Gamma^{\mathcal{M}}_{\alpha+1} = \Gamma^{\mathcal{M}}_{\alpha}$, i.e., $\Gamma^{\mathcal{M}}_{\alpha}$ is a fixpoint of $\Gamma^{\mathcal{M}}$. In general, if $A$ is any fixpoint of $\Gamma^{\mathcal{M}}$, then
	\[(\mathcal{M},A)\models\PT^-.\]
	Let $\alpha_{\mathcal{M}}$ denote the least ordinal $\alpha$ such that $\Gamma_{\alpha}^{\mathcal{M}}$ is a fixpoint of $\Gamma^{\mathcal{M}}$.	
\end{definicja}
In \cite{CWL}, the following lemmata were proved:
\begin{lemat}\label{lemat_recsat_to_omega}
	If $\mathcal{M}\models \PA$ is recursively saturated, then $\alpha_{\mathcal{M}} = \omega$.
\end{lemat}

\begin{lemat}
	If $\mathcal{M}\models \PA$ and $\alpha_{\mathcal{M}}=\omega$, then $(\mathcal{M},\Gamma^{\mathcal{M}}_{\omega})\models \PT^- + \intot$.
\end{lemat}

Now we shall show that the converse to \ref{lemat_recsat_to_omega} holds. In particular, our method of finding extensions for $\PT^- + \intot$ works only for recursively saturated models.

\begin{lemat}
	For every non-standard $\mathcal{M}\models \PA$. If $\alpha_{\mathcal{M}} = \omega$, then $\mathcal{M}$ is recursively saturated.
\end{lemat}
\begin{proof}
	We prove the contraposition: suppose that a non-standard model $\mathcal{M}$ is not recursively saturated. 
	Let $p(x)$ be a recursive type using parameters from $\bar{a}$ which is omitted in $\mathcal{M}$. Let $(\phi_i(x,\bar{y}))_i$ be an arithmetically representable enumeration of formulae in $p(x)$. Without loss of generality, assume that $\phi_0(x,\bar{y}) = (x=x)$. Let 
	\[\psi_i(x,\bar{y}) = \bigwedge_{j< i}\phi_j(x,\bar{y})\wedge\neg \phi_i(x,\bar{y})\]
	Then every $b\in M$ satisfies exactly one of $\psi_i(x,\bar{a})$ (since $p(x)$ is omitted). 
	Now, for every $n\in\omega$ we shall define formulae $\theta_n(x)$ as follows:
	\begin{eqnarray}
	\theta_n^0(x) &=& (x\neq x)\nonumber\\
	\theta_n^{k+1}(x,\bar{y}) &=&\psi_{n-(k+1)}(x,\bar{y})\vee \theta_n^{k}(x,\bar{y})\nonumber\\
	\theta_n(x,\bar{y}) &=& \theta^{n}_n(x,\bar{y})\nonumber					
	\end{eqnarray}
	Let us observe that the above construction can be arithmetized and therefore for some $b\in M\setminus\mathbb{N}$ there exists a (code of a) formula $\theta_b(x,\bar{y})$, which is of the following form:
	\[(\psi_0(x,\bar{y}) \vee (\psi_1(x,\bar{y})\vee (\psi_2(x,\bar{y})\vee (\ldots(\psi_{b-1}(x,\bar{y}) \vee x\neq x)\ldots)\]
	Then for each $c\in M$, there exists $n\in\omega$ such that $\theta_b(\num{c},\bar{\num{a}})\in \Gamma^{\mathcal{M}}_n$, since each $c$ satisfies some $\psi_i(x,\bar{a})$. But also for every $i\in\omega$, there exists $c\in M$ such that the least $n$ for which $\psi_n(c,\bar{a})$ is greater than $i$.
	Consequently, there is no $k\in\omega$ for which 
	\[\theta_b(c,\bar{\num{a}})\in\Gamma_k^{\mathcal{M}}\]
	for every $c\in M$. In particular, $\forall v\theta(v,\bar{\num{a}})\notin \Gamma_{\omega}^{\mathcal{M}}$ and consequently the $\PT^-$ axiom
	$$\forall v \forall \phi(v)\ \ \bigl(T(\forall v\phi)\equiv \forall x T(\phi(\num{x}))\bigr)$$
	is not satisfied in $(\mathcal{M},\Gamma^{\mathcal{M}}_{\omega})$. Hence $\alpha_{\mathcal{M}}\neq \omega$.
\end{proof}

\section{Models of $\pt^- + \inter$}\label{sect_int}

Since we have shown that $\PTtot$ is not a model-theoretically weak theory, as was originally hoped, one could start wondering whether it differs in some significant respect from $\PTint.$ In this section, we will show that actually this is the case. Namely, it turns out that $\PTint$ is still model-theoretically stronger than $\PTtot.$ As we shall see, any model of $\PA$ expandable to a model of $\PTint$, is also expandable to a model of $\UTB.$ We know that any model of $\PA$ expandable to a model of $\UTB$ is recursively saturated and that this  containment is strict, i.e., not every recursively saturated model of $\PA$ admits an expansion to a model of $\UTB.$\footnote{\poprawka{We know that there exist rather classless recursively saturated models of $\PA$, i.e., recursively saturated models $M \models \PA$ with the following property: for every $X \subseteq M$ such that every initial segment of $X$ is coded $M$, the set $X$ is definable in $M$ with an arithmetical formula (with parametres). Since no subset of $M$ definable with an arithmetical formula can satisfy $\UTB^-$, we see that no such model $M$ can admit an expansion to a model of $\UTB$. The existence of recursively saturated, rather classless models has been shown by Kaufmann in \cite{kaufmann} under an additional set-theoretic assumption $\diamond$. The assumption has been dropped by Shelah, \cite{shelah}, Application C, p. 74.}}
 On the other hand, it has been shown in \cite{CWL}, Theorem 3.3 that any recursively saturated model of $\PA$ admits an expansion to a model of $\PTtot.$

\begin{tw} \label{tw_ptint_a_utb}
	Suppose that $(M,T)$ is a model of $\PTint$. Then there exists a $T'$ such that $(M,T') \models \UTB.$
\end{tw}

\begin{proof}
	Let $(M,T) \models \PTint$. We will find $T'$ such that $(M,T') \models \UTB.$ Without loss of generality we may assume that $M$ is nonstandard. As in the previous section, we will use Lemma \ref{lem_total_stopping_disjunction}. Let us fix any primitive recursive enumeration $(\phi_i)_{i =0}^{\infty}$ of arithmetical formulae. Then let 
	\begin{displaymath}
	\alpha'_i(\phi,t)
	\end{displaymath}
	 be defined as the (formalised version of the) formula "$t$ is a (finite) sequence of terms $(t_1, \ldots, t_n)$ and $\phi = \phi_i(t_1, \ldots, t_n)$" and let 
	 \begin{displaymath}
	 \alpha_i(\phi,t,b) = \alpha_i'(\phi,t) \vee \num{i}>b.
	 \end{displaymath}
	Let 
	\begin{displaymath}
	\beta'_i(t)
	\end{displaymath}
	be defined as ''$t$ is a (finite) sequence of terms $t_1, \ldots, t_n$ and $\phi_i(t_1, \ldots,t_n)$.'' Let
	\begin{displaymath}
	\beta_i(t,b)
	\end{displaymath}
	 be $\beta'_i(t) \wedge \num{i} \leq b.$ Note that $\phi$ is not a free variable of the formula $\beta_i.$ 
	Let us fix any nonstandard $c \in M$ and let
	\begin{displaymath}
	\tau(\phi,t,b) = \bigvee_{i=0}^{\alpha(\phi,t,b), c} \beta_i(t).
	\end{displaymath}
	
	Note that for any standard $c$ the predicate $\tau$ is equivalent to the very simple arithmetical truth predicate:
	\begin{displaymath}
	\tau_n(\phi,t,b) = \bigvee_{i=0}^{n} \phi= \phi_i(t) \wedge \phi_i(t) \wedge \num{i}<b.
	\end{displaymath}
	At this point one may wonder, what is the role of the variable $b.$ It is indeed technical. We artificially 
	truncate our truth predicates so that they work only for the first $b$ formulae. \poprawka{This is to some extent controlled by the parameter $c$ in the definition of $\tau$, since whenever $c$ is standard, the formula $\tau$ works like a truth predicate only for the first $c$ sentences. However, $c$ is not a variable in the formula $\tau$, but rather a parameter describing the syntactic shape of $\tau$, whereas we need this truncation to be expressed with a variable for reasons which will shortly become clear.}  
	
	It turns out that for some parameter $b$ the formula given by
	\begin{displaymath}
	T'(\phi) =\exists t \ T (\tau(\num{\phi},\num{t},\num{b}))
	\end{displaymath} 
	 satisfies $\UTB$. We will prove this claim in a series of lemmata. This will obviously conclude our proof. 
\end{proof}

\begin{lemat} \label{lem_utb-}
	Let $\tau'(\phi,t) = \tau(\phi,t,b)$ for some fixed nonstandard $b$. Then for an arbitrary standard arithmetical formula $\phi(v_1,\ldots,v_n)$ and an arbitrary sequence of terms $t=(t_1,\ldots, t_n)$,  possibly nonstandard (the length of the sequence is assumed to be standard).
	\begin{displaymath}
	(M,T) \models T\tau'(\num{\phi(t_1, \ldots, t_n)},\num{t}) \equiv \phi(t_1, \ldots, t_n).
	\end{displaymath}
\end{lemat}
\begin{proof}
	If $\phi$ is standard, then $\phi = \phi_i$ for some standard $i.$ So by Lemma \ref{lem_total_stopping_disjunction}
	\begin{displaymath}
	(M,T) \models T\tau'(\num{\phi(t_1, \ldots, t_n)},\num{t}) \equiv \beta_i(t) \equiv \phi_i(t_1, \ldots, t_n),
	\end{displaymath}
	which is exactly the claim of the lemma.
\end{proof}

Note that the above lemma is true in pure $\PT^-$. We have used no induction at all. Now we only need to check that for some parameter $b$ the predicate $T'(\phi, t)$ defined as $T \tau(\num{\phi}, \num{t},\num{b})$ is fully inductive.

\begin{lemat} \label{lem_tot_cons}
	Let $T'$ be defined as in the above proof. Then for some $b$, the formula $\tau'(\phi,t) = \tau(\phi,t,\num{b})$ 
	is total and consistent i.e. for all $\phi$ and $t$, exactly one of $T\tau'(\num{\phi},\num{t}),$  $T \neg \tau'(\num{\phi}, \num{t})$ holds.
\end{lemat}
\begin{proof}
	Note first that for any standard $b$, the formula $\tau(\phi,t,\num{b})$ is total and consistent. Namely, since $\alpha_i(\phi,t,n)$ is true for any $i > n$,
	 we see that for any $\phi,t$ the least $i$ such that $\alpha_i(\phi,t,n)$ holds is standard (it is at most $n+1$) and then the assumptions of Lemma \ref{lem_total_stopping_disjunction} are satisfied. This implies that for any fixed $\xi$ the formula $T\tau(\num{\xi},\num{t},\num{n})$ is equivalent to some $\phi_i(\val{t}) \wedge \num{i} \leq\num{n}$,
	  which is a standard formula. This implies that for any $t$, exactly one of $T\tau(\num{\xi},\num{t},\num{n}), T\neg \tau(\num{\xi},\num{t},\num{n})$ holds. 
	
	Now, consider the formula
	\begin{displaymath}
	\psi(b) = \forall \phi,t \ \Big(\tau(\phi,t,b) \vee \neg \tau(\phi,t,b)\Big).
	\end{displaymath}
	We have just shown that for an arbitrary standard $n$ we have $T\forall b<\num{n} \  \psi(b).$
	 So by internal induction we have for some nonstandard $d_1$
	\begin{displaymath}
	T \bigg(\forall b\leq \num{d_1} \forall \phi,t \ \Big(\tau(\phi,t,b) \vee \neg \tau(\phi,t,b)\Big)\bigg),
	\end{displaymath}
	which gives
	\begin{displaymath}
	\forall b\leq d_1 \forall \phi,t \ \Big(T\tau(\num{\phi},\num{t},\num{b}) \vee T\neg \tau(\num{\phi},\num{t},\num{b})\Big)
	\end{displaymath}
	
	Similarly, let 
	\begin{displaymath}
	\xi(b) =  \exists \phi,t \ \Big(\tau(\phi,t,b) \wedge \neg \tau(\phi,t,b)\Big).
	\end{displaymath}
	Suppose that $T (\exists d<\num{b} \ \xi(d))$ holds for any nonstandard $b<d_1$. Then by underspill we would have $T\xi(\num{n})$ for some $n \in \omega.$ But we have just shown that this is impossible. So there exists some nonstandard $b<d_1$ such that for any $d \leq b$ and any $\phi,t$ at most one of $T\tau(\num{\phi},\num{t},\num{d}), T \neg \tau(\num{\phi},\num{t},\num{d})$ holds. At the same time, we know that at least one of these formulae holds. So $\tau'(\phi,t) = \tau(\phi,t,b)$ is total and consistent. 
\end{proof}

We are very close to showing that we have defined a predicate satisfying full induction. Before we proceed, we have to introduce some new notation. Let $\eta$ be any formula containing a unary predicate $P$ not in the language of $\PT^-$ and let $\xi(v)$ be an arbitrary formula with one free variable. Then by $\eta[\xi/P]$ (or simply $\eta[\xi]$) we mean a formula resulting from substituting $\xi(v_i)$ for any instance of $P(v_i)$ in $\eta$. We assume that all the variables in $\eta$ has been renamed so as to avoid clashes.

Let us give an example. Let $\eta(x,y) = P(x+y) \wedge \exists z \ (z=y \wedge P(z))$. Let $\xi(v) = (v>0)$. Then 
\begin{displaymath}
\eta[\xi] = x+y>0 \wedge \exists z \ (z=y \wedge z>0).
\end{displaymath}

Now, basically, we would like to finish the proof in the following way. Let $\tau'$ be a total formula defined as in the above lemmata and let $\eta$ be an arbitrary standard formula from the arithmetical language enlarged with a fresh unary predicate $P$. Then, applying compositional axioms a couple of times, we see that
\begin{displaymath}
T(\eta[\tau]) \equiv \eta[T \tau].
\end{displaymath}
Let us call this principle the \df{generalised commutativity}. If this were true, then we could conclude our proof. Namely, by the internal induction principle, we know that
\begin{displaymath}
\bigl(\forall x \bigl(T(\eta[\tau](\num{x}))\rightarrow T(\eta[\tau](\num{x+1}))\bigr)\longrightarrow \bigl(T(\eta[\tau](\num{0}))\rightarrow \forall x T(\eta[\tau](\num{x}))\bigr)
\end{displaymath}
which, by generalised commutativity, would allow us to conclude that
\begin{displaymath}
\bigl(\forall x \bigl(\eta[T\tau](\num{x})\rightarrow (\eta[T\tau](\num{x+1})\bigr)\bigr)\longrightarrow \bigl(\eta[T\tau](\num{0})\rightarrow \forall x \eta[T\tau](\num{x})\bigr).
\end{displaymath}
Since the choice of $\eta$ was arbitrary, this precisely means that $\tau$ satisfies the full induction scheme. 

\poprawka{The generalised commutativity principle in the form stated above does not even quite make sense, since we would have to apply the truth predicate to a formula containing free variables. Therefore, we have to restate it in somewhat more careful manner. }

\begin{definicja}
	Fix a unary predicate $P$. Let $\eta$ be an arbitrary formula from the language containing that predicate. We say that $\eta$ is in \df{semirelational} form if the predicate $P$ is applied only to variables rather than to arbitrary terms.
\end{definicja}

We may always assume that formulae we \emph{use} are semirelational, \poprawka{since we may eliminate any occurrence of $P(t)$ for complex terms $t$, by replacing it with $\exists x \ (x = t \wedge P(x))$. This is expressed in the following lemma:}

\begin{lemat} \label{lem_semirelational}
	Any formula is equivalent in first-order-logic to a formula in semirelational form.
\end{lemat}

Now we are ready to state generalised commutativity lemma in a proper manner. 

\begin{lemat} \label{lem_gen_commutativity}
	Let $(M,T) \models \PT^-.$ Let $T * \xi(x) = T(\xi(\num{x}))$ for every $x.$ Suppose that $\xi$ is total and consistent. Let $\eta$ be an arbitrary standard formula from the arithmetical language extended with a fresh unary predicate $P$. Then the formula $\eta[\xi]$ is total and consistent, and
	\begin{displaymath}
	(M,T) \models \forall x_1, \ldots, x_n \big( T(\eta[\xi](\num{x_1},\ldots, \num{x_n})) \equiv \eta[T * \xi](x_1, \ldots, x_n)\big).
	\end{displaymath}
\end{lemat}

The lemma generalises to the case, where the predicate $P$ is not unary (i.e. $\xi$ may have more than one variable). The proof may be easily adapted to cover this case. We will actually use the lemma for the case with $P$ binary.

\begin{proof}
	We prove both claims simultanously by induction on complexity of $\eta.$ Suppose that $\eta$ is an atomic formula. Then it is either of the shape $s=t$ for some standard arithmetical terms $s$,$t$, or of the form $P(x)$.
	
	In the first case, $\eta[\xi] = \eta$, and the following equivalences hold:
	\begin{eqnarray*}
		T(s(\num{x_1}, \ldots, \num{x_n})=t(\num{x_1}, \ldots, \num{x_n})) & \equiv & s(\num{x_1}, \ldots, \num{x_n})^{\circ}=t(\num{x_1}, \ldots, \num{x_n})^{\circ} \\
		& \equiv & s(x_1, \ldots, x_n) = t(x_1, \ldots, x_n) \\
		& = & \eta[T*\xi](x_1, \ldots, x_n).
	\end{eqnarray*}
	
	If $\eta = P(x)$, then $\eta[\xi] = \xi$ and
	\begin{displaymath}
	T(\eta[\xi](\num{x})) =  T \xi(\num{x}) = \eta[T * \xi](x).
	\end{displaymath}
	
	So let prove the induction step. If $\eta$ is a conjunction or disjunction, then the proof is straightforward (the fact that a conjunction or disjunction of sentences which are either true or false is itself either true or false is an easy application of the compositional axioms of $\PT^-$). If $\eta = \neg \rho$, then we know by induction hypothesis that $\rho[\eta]$ is total and consistent. Then by the compositional axiom for double negation for the truth predicate, the formula $\neg \rho[\eta]$ is also total and consistent and the following equivalences hold:
	\begin{eqnarray*}
		T(\neg \rho[\eta](\num{x_1}, \ldots, \num{x_n})) & \equiv & \neg T(\rho[\eta](\num{x_1}, \ldots, \num{x_n})) \\
		& \equiv & \neg \rho[T * \eta](x_1, \ldots, x_n).
	\end{eqnarray*}
	
	The induction step for quantifier axioms is also simple. Let us prove it for the existential quantifier. Suppose that $\eta = \exists x \ \rho(x,x_1, \ldots, x_n).$ Then
	\begin{eqnarray*}
		T(\exists x \ \rho[\xi](x, \num{x_1}, \ldots, \num{x_n})) & \equiv & \exists x \ T(\rho[\xi](\num{x}, \num{x_1}, \ldots, \num{x_n}) \\
		& \equiv & \exists x \ \rho[T * \xi](x, x_1 \ldots,x_n) \\
		& = & \eta[T * \xi](x_1, \ldots, x_n).
	\end{eqnarray*}
	The second equivalence follows by the induction hypothesis and the last equality by definition. So let us check that $\eta[\xi]$ is total and consistent. Suppose that $T(\exists x \rho[\xi](x, \num{x_1}, \ldots, \num{x_n}))$ does not hold. Then by compositional axioms for the truth predicate, there is no $x$ such that
	\begin{displaymath}
	T(\rho(\num{x}, \num{x_1}, \ldots, \num{x_n})).
	\end{displaymath}
	By induction hypothesis, $\rho$ is total and consistent, so for all $x$ we must have
	\begin{displaymath}
	T (\neg \rho[\xi](\num{x}, \num{x_1}, \ldots, \num{x_n})).
	\end{displaymath}
	This entails, again by compositional clauses
	\begin{displaymath}
	T (\neg \exists x \ \rho[\xi](x, \num{x_1}, \ldots, \num{x_n})).
	\end{displaymath}
\end{proof}

Now we are ready to conclude the proof of our theorem.

\begin{lemat} \label{lem_indukcja_tauprim}
	Let $(M,T)$ be any nonstandard model of $\PTint.$ Suppose that $\tau'(\phi,t)$ satisfies the claim Lemma \ref{lem_tot_cons}.
	 Then the predicate $T'(\phi,t)$ defined as $T * \tau'(\phi,t)$ satisfies the full induction scheme. 
\end{lemat}
\begin{proof}
	By internal induction principle, the following holds for an arbitrary standard 
	 $\eta$ from the arithmetical language extended with one fresh unary predicate $P(v)$:
	\begin{displaymath}
	\Bigl(\forall x \bigl(T(\eta[\tau](\num{x}))\rightarrow T(\eta[\tau](\num{x+1})\bigr)\Bigr)\longrightarrow \Bigl(T(\eta[\tau](\num{0}))\rightarrow \forall x T(\eta[\tau](\num{x}))\Bigr),
	\end{displaymath}
	Since $\tau'$ is total, if we additionally assume that $\eta$ is semirelational,
	we can reach the following conclusion by Lemma \ref{lem_gen_commutativity}:
	\begin{displaymath}
	\Bigl(\forall x \bigl(\eta[T * \tau](x)\rightarrow (\eta[T * \tau](x+1)\bigr)\Bigr)\longrightarrow \Bigl(\eta[T * \tau](0)\rightarrow \forall x \eta[T * \tau](x)\Bigr).
	\end{displaymath}
	Since $\eta$ was an arbitrary semirelational formula and any formula is equivalent to a semirelational one, this shows that $T'$ satisfies the full induction scheme.
\end{proof}

\begin{proof}[The conclusion of the proof of Theorem \ref{tw_ptint_a_utb}] 
	We have defined a formula $T * \tau'(\phi,t)$ which satisfies full induction scheme and such that for an arbitrary standard $\phi(v_1,\ldots,v_n)$ and an arbitrary sequence of terms $(t_1,\ldots, t_n)$ the following holds:
	\begin{displaymath}
	(M,T) \models T * \tau'(\phi(t_1,\ldots, t_n),t) \equiv \phi(t_1, \ldots, t_n).
	\end{displaymath}
	Then the formula  $T'(\phi)$ defined as 
	\begin{displaymath}
	\exists t \  T*\tau'(\phi,t) 
	\end{displaymath}
	satisfies the uniform disquotation axioms of $\UTB$ as well as the full induction scheme. So it defines a predicate satisfying $\UTB$ in $(M,T).$
\end{proof}

This model-theoretic result allows us to make some conclusions concerning relative definability of the introduced theories.
\begin{wniosek}
	The theory $\PTtot$ does not relatively truth defines $\PTint.$
\end{wniosek}
\begin{proof}
	We have just checked that every model $(M,T) \models \PTtot$ may be expanded to a model of $\UTB$. By Theorem TODO, there exist recursively saturated rather classless models which cannot be expanded to any model of $\UTB.$ On the other hand in \cite{CWL}, Theorem 3.3, it has been shown that any recursively saturated model can be expanded to a model of $\PTtot.$ Thus, there exist models of $\PTtot$ which cannot be expanded to a model of $\PTint$. This contradicts relative definability.
\end{proof}

\section{Weak and Expressive Theories  of Truth}\label{sect_instrum}

In \cite{fischerhorsten}, the authors searched for a theory of truth that would simultaneously satisfy two requirements:\poprawka{
\begin{enumerate}
	\item It could model the use of truth in model theory;
	\item It would witness the expressive function of the notion of truth.
\end{enumerate}}

The way to satisfy the former is to be model-theoretically conservative over $\PA$. Being such, the theory would not discriminate among possible interpretations of our basis theory. The way to satisfy the latter is to allow for expressing "thoughts" which are not expressible in the basis theory. There are many ways in which a theory of truth can witness the expressive role of the notion of truth. To mention just two (for the rest of the examples the Reader should consult \cite{fischerhorsten}): if a theory of truth is finitely axiomatizable, then it is more expressive (than $\PA$)
and if a theory of truth has non-elementary speed-up over $\PA$, then it is more expressive (than $\PA$).
There is a canonical construction which produces a theory of truth satisfying both  finite-axiomatizability and the speed-up desideratum:
the theory has to be (at least partially) classically compositional and it has to prove
that all standard instantiations of the induction scheme are true. Without aspiring to any sort of completeness, let us offer the following explication of both properties. We start with a useful definition:
\begin{definicja}
	Let $\CC(x)$ denote the disjunction of the following formulae
	\begin{itemize}
		\item $\exists s,t\ \ \bigl(x = (s=t)\wedge (T(x)\equiv \val{s}=\val{t})\bigr)$
		\item $\exists \phi,\psi\ \ \bigl(x = \phi\vee\psi \wedge (T(x)\equiv T(\phi)\vee T(\psi))\bigr)$
		\item $\exists \phi \ \ \bigl( x=\neg\phi \wedge (T(x)\equiv \neg T(\phi))\bigr)$
		\item $\exists \phi\exists v\ \ \bigl(x = \exists v\phi \wedge (T(x)\equiv \exists y T(\phi(\num{y})))\bigr)$
		\item $\form^{\leq 1}(x) \wedge \forall s,t\Bigl( \val{s} = \val{t}\rightarrow \bigl(T(x(s))\equiv T(x(t))\bigr)\Bigr)$
	\end{itemize}
\end{definicja}

Informally, $\CC(x)$ says that $x$ is a formula on which $T$ behaves compositionally in the sense of classical first-order logic.

\begin{definicja}
	A truth theory $\Th$ is \textit{partially classically compositional} if there exists a formula $D(y)$ such that $\Th$ proves the following sentences:
	\begin{enumerate}
		\item $\forall y (D(y)\rightarrow \forall x\leq y D(x))$;
		\item $D(0)\wedge \forall y (D(y)\rightarrow D(y+1))$;
		\item $\forall y \bigl(D(y)\rightarrow \forall \phi (\dpt(\phi)\leq y \rightarrow \CC(\phi))\bigr)$;
	\end{enumerate}
	where $\dpt(\phi)\leq x$ denotes an arithmetical formula representing the (primitive recursive) relation "the \poprawka{ depth of the syntactic tree of $\phi$} is at most $x$". 
\end{definicja}
If a formula satisfies the first requirement, we say that it is \emph{downward closed}. If a formula satisfies the second one, we say that it is \emph{progressive}. If a formula $D(y)$ is both downward closed and progressive, we will say that it \textit{defines an initial segment}. This is justified, since if
$D(y)$ satisfies $1$ and $2$, then in each model $\mathcal{M}\models\Th$ the set $\set{a\in M}{\mathcal{M}\models D(a)}$ is an initial segment of the model. In fact, being \emph{downward closed} is not a very restrictive condition: if $D(y)$ is progressive, then the formula
\[D'(x):= \forall y\leq x D(y)\] 
defines an initial segment. (this corresponds to a model-theoretic fact that each \emph{cut} can be \emph{shortened} to an initial segment). The third condition says that if $\phi$ is not too complicated (i.e., its complexity belongs to the initial segment defined by $D$), then $T$ behaves classically on $\phi$.

\begin{definicja}
	Let $\ind(\phi(x))$ denote the instantiation of the induction scheme with $\phi(x)$, i.e., the universal closure of the following formula:
	\[\forall x (\phi(x)\rightarrow\phi(x+1))\longrightarrow \bigl(\phi(0)\rightarrow \forall x\phi(x)\bigr)\]
	Following our conventions, we will use $\ind(\cdot)$ to denote an arithmetical formula representing the function which, given a G\"odel code of a formula with at most one free variable, returns the G\"odel code of the corresponding
	induction axiom. 
\end{definicja}

\begin{definicja}
	A truth theory \textit{proves the truth of induction} if there exists a formula $D(y)$ such that $\Th$ proves that $D(y)$ defines an initial segment and 
	\begin{equation}\label{truth_ind}\tag{T(IND)}
	\forall \phi(v)\ \ \biggl(D(\phi(v))\rightarrow T\bigl(\ind(\phi(v))\bigr)\biggr).
	\end{equation}
\end{definicja}

We shall say that $\Th$ is finitely axiomatisable modulo $\PA$ if there is a sentence $\phi$ such that the logical consequences of $\Th$ are precisely the logical consequences of $\PA\cup\was{\phi}$. For example, $\CT^-$, $\PT^-$ and $\WPT^-$ are finitely axiomatisable modulo $\PA$.

Now we have the following theorem whose unique novelty rests on isolating the features that are usually used to prove the thesis for concrete theories of truth.

\begin{tw}
	Assume that
	\begin{enumerate}
		\item $\Th$ is partially classically compositional and proves the truth of induction and
		\item $\Th$ is finitely axiomatizable modulo $\PA$,
	\end{enumerate} 
	then $\Th$ is finitely axiomatizable and it has super-exponential speed-up over $\PA$.
\end{tw}
\begin{proof}[Sketch of the proof]
	Let $D_1(y)$ define an initial segment on which $T$ is classically compositional. Let $D_2(y)$ define an initial segment on which $\Th$ proves the truth on induction. Then $D(y): = D_1(y)\wedge D_2(y)$ defines an initial segment on which $T$ is classically compositional and proves the truth of induction. Obviously, for every standard natural number $n$ we have 
	\[\Th \vdash D(\num{n})\]
	In particular, $\Th\vdash \UTB^-$,
	which for every concrete formula of standard complexity can be proved by external induction on the complexity of its subformulae. Now, for every standard formula $\phi(x_0,\ldots,x_n)$, we can prove $\ind(\phi(\bar{x}))$ in $\Th$ in the following way: 
	\begin{enumerate}
		\item prove that $D$ defines an initial segment on which $T$ is classically compositional;
		\item prove that $D(\qcr{\ind(\phi(\bar{x}))})$ and conclude $D(\qcr{\phi(\bar{x})})$;
		\item prove \ref{truth_ind};
		\item using $2.$ and $3.$ conclude $T\left(\qcr{\ind(\phi(\bar{x}))}\right)$;
		\item prove $\UTB^-\left(\ind(\phi(\bar{x}))\right)$;
		\item conclude $\ind\left(\phi(\bar{x})\right)$.
	\end{enumerate}
	Observe that, given $1.$, the proof of $D(\qcr{\ind(\phi(\bar{x}))}$ can be constructed in pure First-Order Logic. Similarly, given $1.$, all we need to use in proving $\UTB^-\left(\ind(\phi(\bar{x}))\right)$ are some basic syntactical facts provable in $I\Sigma_1$. Let $\phi$ be a sentence such that $\PA\cup\was{\phi}$ is a finite-modulo-$\PA$ axiomatisation of $\Th$. It follows that for some $n$, every proof \dousuniecia{(every proof, czy proof of every?)} of $\ind(\phi(\bar{x}))$ can be given in $I\Sigma_n + \phi$, hence
	the theory
	\[I\Sigma_n\cup\was{\phi}\]
	is a finite axiomatization of $\Th$. To prove that $\Th$ has super-exponential speed-up over $\PA$, we show that there is a formula $D'(y)$ which provably in $\Th$ defines an initial segment and that
	\[\Th\vdash \forall y\bigl( D'(y)\rightarrow \ConPA(y)\bigr)\]
	where $\ConPA(y)$ is a finitary statement of consistency of $\PA$ saying that there is no $\PA$ proof of $0=1$ which can be coded using less than $y$ bits. For the details, see \cite{fischer2}, Theorem 9.
\end{proof}

In \cite{fischer2}, it was shown that $\PT^- + \intot$ satisfies the assumptions of the above theorem. However, as was shown in \cite{CWL}, this theory is not model-theoretically conservative over $\PA$. We shall now show that the right theory to use is $\WPT^-+\inter$.

\begin{stwierdzenie}\label{prop_wpt_spid}
	$\WPT^-$ is partially classically compositional and $\WPT^- + \inter$ proves the truth of induction.
\end{stwierdzenie}

\begin{proof}
	Let us define
	\[D'(y):= \forall x\leq y\forall \phi \biggl(\dpt(\phi)\leq x\longrightarrow \bigl(T(\neg\phi)\equiv \neg T(\phi)\bigr)\biggr).\]
	\dousuniecia{Usunąłem z powyższego różki.}
	Then it can be easily shown that $D'(y)$ provably in $\WPT^-$ defines an initial segment on which $T$ is classically compositional (that $D'(y)$ is progressive is assured by the compositional axioms of $\WPT^-$). For convenience, let us define:\footnote{\poprawka{$\GC$ stands for ''generalised commutativity.'' $\GC(x)$ expresses that the truth predicate commutes with the whole block of universal quantifiers in the universal closure of $x$.}}
	\[\GC(x):= \form(x) \wedge \bigl(T(\ucl(x))\equiv \forall \sigma (\Asn(x,\sigma)\rightarrow T(x[\sigma]))\bigr)\]
	where 
	\begin{enumerate}
		\item  $\ucl(\phi(\bar{x}))$ denotes the universal closure of $\phi(\bar{x})$;
		\item $\Asn(\phi,\sigma)$ represents the relation "$\sigma$ is an assignment for $\phi$", i.e., $\sigma$ is \poprawka{a function} defined exactly on the free variables of $\phi$;
		\item $x[\sigma]$ denotes the result of simultaneous substitution of numerals naming numbers assigned by $\sigma$ to the free variables of $x$.
	\end{enumerate} 
	Further define
	\[D(y):= D'(y)\wedge \forall \phi(\bar{x})\biggl(|\FV(\phi(\bar{x}))|\leq y\longrightarrow \GC(\phi(\bar{x}))\biggr)\]
	where, , $|\FV(\phi\bar{x})|\leq x$ represents the relation "$\phi(\bar{x})$ contains at most $x$ free variables".
	For the sake of definiteness, we assume that $\ucl(\phi)$ starts with a quantifier binding the variable with the least index among the free variables if $\phi$. It can be easily seen that $D(y)$ is downward closed. Let us now show that $D(y)$ is progressive. We work in $\WPT^-$. Fix arbitrary $a$ and suppose that $D(a)$. Then $D'(a)$ and, as $D'(y)$ is progressive, we have also $D'(a+1)$. Let us fix an arbitrary formula $\phi$ with less than $a+1$ free variables and let $v$ be its free variable with the least index. Then the following are equivalent
	\begin{enumerate}
		\item $T(\ucl(\phi))$
		\item $\forall x T(\ucl(\phi(\num{x}/v)))$
		\item $\forall x\forall \sigma \bigl(\Asn(\phi(\num{x}/v),\sigma)\rightarrow T(\phi(\num{x}/v)[\sigma])\bigr)$
		\item $\forall \sigma \bigl(\Asn(\phi,\sigma)\rightarrow T(\phi[\sigma])\bigr)$.
	\end{enumerate}  
	The equivalence between $1.$ and $2.$ is by the axiom for universal quantifier in $\WPT^-$. The equivalence between $2.$ and $3.$ holds because $\phi(\num{x}/v)$ has $\leq a$ free variables. The last equivalence holds because each assignment for $\phi$ consists of an assignment to $v$ and an assignment to the free variables of $\phi(\num{x}/v)$. 
	
	We show that $\WPT^-+\inter$ proves the truth of induction on $D(y)$. We work in  $\WPT^-+\inter$. Let us observe that for each formula $\phi$, we have $\dpt(\phi)\leq \phi$ and $|\FV(\phi)|\leq \phi$. \footnote{Being precise, this is a property of our coding. But most natural codings surely have it.}
	Hence if $D(\phi)$, then $D'(\dpt(\phi))$ and $D(|\FV(\phi)|)$. In particular, if $D(\phi)$ then $T$ is classically compositional on subformulae of $\phi$ and $T$ can successfully deal \dousuniecia{(Napisałbym, co to znaczy ,,succesfully deal'' w tym kontekście.)} with the universal closure for $\phi$.
	Let us fix an arbitrary formula $\phi(v,\bar{w})$ such that $D(\phi(v,\bar{w}))$. We have to show $T(\ind(\phi(v,\bar{w}))$, i.e.
	\begin{equation}\label{equat}
	T\biggl(\ucl\Bigl(\forall v (\phi(v)\rightarrow\phi(v+1))\longrightarrow \bigl(\phi(0)\rightarrow \forall v\phi(v)\bigr)\Bigr)\biggr)
	\end{equation}
	(we skip the reference to $\bar{w}$ and assume that they are bounded by the universal quantifiers from $\ucl$). Since the formula 
	$$\bigl(\forall v (\phi(v)\rightarrow\phi(v+1))\longrightarrow \bigl(\phi(0)\rightarrow \forall v\phi(v)\bigr)\bigr)$$
	 contains less free variables then $\phi(v)$, we know that \eqref{equat} is equivalent to 
	\begin{equation}\label{equat2}
	\forall \sigma T\Bigl(\forall v (\phi(v)[\sigma]\rightarrow\phi(v+1)[\sigma])\longrightarrow \bigl(\phi(0)[\sigma]\rightarrow \forall v\phi(v)[\sigma]\bigr)\Bigr).
	\end{equation}
	Let us fix an arbitrary $\sigma$. Then $\phi(v)[\sigma]$ is a formula with at most free variable $v$. Let us abbreviate it with $\psi(v)$. Hence it is enough to show:
	\begin{equation}\label{equat3}
	T\Bigl(\forall v (\psi(v)\rightarrow\psi(v+1))\longrightarrow \bigl(\psi(0)\rightarrow \forall v\psi(v)\bigr)\Bigr).
	\end{equation}
	Since $\dpt(\psi(v)) = \dpt(\phi(v))$ and the depth of \eqref{equat3} is equal to $\dpt(\psi(v)) + 3$, then $T$ is classically compositional on \eqref{equat3}. Hence \eqref{equat3} is equivalent to 
	\[\forall x \bigl(T(\psi(\num{x}))\rightarrow T(\psi(\num{x+1}))\bigr)\longrightarrow \bigl(T(\psi(\num{0}))\rightarrow \forall x T\psi(\num{x})\bigr)\]
	which follows by $\inter$. Hence $\WPT^-+\inter$ proves the truth of induction on $D$.
\end{proof}

Hence $\WPT^-+ \inter$ exemplifies the expressive role of truth. Let us observe that, as it contains no restriction on arithmetical formulae admissible in the axiom of internal induction, it is more natural than $\PT^-+\intot$.\footnote{In \cite{fischerhorsten} authors discuss this restriction in the context of $\PT^-$ and admit it is as a possible objection to their theory.} $\WPT^-+ \inter$ proves that all arithmetical formulae, and not only total, satisfy induction, which is clearly the idea behind $\PA$. Let us show that despite having such an expressive axiom, it is a model-theoretically conservative theory of truth.

\begin{tw}\label{tw_wpt_cons}
	$\WPT^- + \inter$ is model-theoretically conservative over $\PA$.
\end{tw}

\begin{proof}
	Let $\mathcal{M}\models \PA$. For $b\in M$ let $b\in Tr'$ if and only if for some $t_0,\ldots, t_n$ such that 
	\[ \mathcal{M}\models\bigwedge_{i\leq n}\term(t_i) \] 
	and some (standard!) $\phi(x_0,\ldots, x_n)\in\mathcal{L}_{\PA}$ 
	\[\mathcal{M}\models \bigl(b = \qcr{\phi(t_0,\ldots, t_n)}\bigr) \wedge \phi(\val{t_0},\ldots,\val{t_n})\]
	Let us observe that with such a definition, we have 
	\[(\mathcal{M}, Tr')\models \UTB^-\]
	To become the appropriate interpretation of $\WPT^-$ truth predicate, $Tr'$ requires only one small correction. Let $\sim_{\alpha}$ denote the arithmetical formula representing in $\PA$ the relation of two \emph{sentences} being the same modulo  \poprawka{renaming}  variables ($\alpha$-conversion).
	Let us define
	\[b \in Tr\]
	if and only if for some $\psi\in Tr'$, $\mathcal{M}\models b\sim_{\alpha}\psi$. Now it can be easily shown that
	\[(\mathcal{M},Tr)\models \WPT^- + \inter.\]
	Indeed, compositional axioms are satisfied, since for every $x\in M$ such that $\mathcal{M}\models \form^{\leq 1}(x)$
	\begin{equation}\label{cond}\tag{$*$}
	(\mathcal{M}, Tr)\models \tot(x) \textnormal{ if and only if for some } n \in\omega\text{, } \mathcal{M}\models \dpt(x)\leq n 
	\end{equation}
	and moreover $(\mathcal{M}, Tr)\models \UTB^-$. Hence in verifying compositional axioms we may use the fact that $\models$ is compositional. Let us check the axiom for $\vee$. Suppose $\phi = \psi\vee\theta$ and $(\mathcal{M}, Tr)\models T(\phi)$. Then there exists $\phi'\sim_{\alpha} \phi$ such that $T(\phi')$ and $\phi' =  \qcr{\phi''(t_0,\ldots,t_n)}$ for some standard $\mathcal{L}_{\PA}$ formula $\phi''(x_0,\ldots,x_n)$ and $t_0,\ldots,t_n$ terms in the sense of $\mathcal{M}$. If so, then $\phi' = \psi'\vee \theta'$ such that $\psi\sim_{\alpha}\psi'$ and $\theta\sim_{\alpha}\theta'$. Also $\psi'$ and $\theta'$ are of the form $\psi''(t_0,\ldots,t_n)$ and $\theta''(t_0,\ldots, t_n)$, respectively. By $\UTB^-$, we have 
	\[\mathcal{M}\models \psi''(\val{t_0},\ldots,\val{t_n})\vee \theta''(\val{t_0},\ldots,\val{t_n})\]
	Without loss of generality, assume that $\mathcal{M}\models \psi''(\val{t_0},\ldots, \val{t_n})$. It means that $(\mathcal{M},Tr)\models T(\psi)$ and consequently $(\mathcal{M}, Tr)\models T(\psi)\vee T(\theta)$. By \eqref{cond}, we have 
	\[(\mathcal{M}, Tr)\models \tot(\psi) \wedge \tot(\theta) \wedge\bigl(T(\psi)\vee T(\theta)\bigr).\]
	which completes the proof of one implication. Let us now assume that the above holds. Since we have $\tot(\psi)$ and $\tot(\phi)$, it follows that for some $n$,$k$,
	\[\mathcal{M}\models \dpt(\psi)\leq n \wedge \dpt(\theta)\leq k.\] 
	In particular, $\dpt(\phi) \leq \max\was{n,k}+1$. Let us assume that $(\mathcal{M},Tr)\models T(\psi)$. Let $\psi\sim_{\alpha} \psi'$ and $\theta\sim_{\alpha}\theta'$ be such that $(\mathcal{M},Tr')\models T(\psi')$.
	Reasoning as previously, we conclude that $(\mathcal{M},Tr')\models T(\psi'\vee \phi')$ 
	and hence 
	\[(\mathcal{M}, Tr)\models T(\psi\vee\phi)\]
	which completes the proof of the compositional axiom for $\vee$.
	
	Let us now verify that $(\mathcal{M}, Tr)\models \inter$. Fix an arbitrary formula $\phi(x)$ in the sense of $\mathcal{M}$ and assume that
	\[(\mathcal{M}, Tr)\models T(\phi(0))\wedge \forall x \bigl(T(\phi(\num{x}))\rightarrow T(\phi(\num{x+1}))\bigr)\]
	It follows that $\mathcal{M}\models \dpt(\phi(x))\leq n$ for some $n\in\omega$ and for some standard 
	\[\phi_0'(x,y_0,\ldots,y_{k}), \phi_1'(x,y_0,\ldots,y_{k}), \phi_2'(x,y_0,\ldots,y_{k}),\] 
	we have
	\[\mathcal{M}\models \phi(x)\sim_{\alpha} \qcr{\phi_i'(x,t_0,\ldots,t_k)}\]
	for some terms in $t_0,\ldots, t_{k}$ in the sense of $\mathcal{M}$ and for $i\leq 2$. In particular, by $\UTB^-$ we have
	\[\mathcal{M}\models \phi'_0(0,\val{t_0},\ldots, \val{t_k}) \wedge \forall x \bigl(\phi'_1(x,\val{t_0},\ldots,\val{t_n})\rightarrow \phi'_2(x+1,\val{t_0},\ldots,\val{t_n})\bigr).\]
	But as satisfiability in a model is closed under $\alpha$-conversion and each two of $\phi'_0$, $\phi_1'$, $\phi_2'$ are $\alpha$-equivalent, we get that 
	\[\mathcal{M}\models \phi'_0(0,\val{t_0},\ldots, \val{t_k}) \wedge \forall x \bigl(\phi'_0(x,\val{t_0},\ldots,\val{t_n})\rightarrow \phi'_0(x+1,\val{t_0},\ldots,\val{t_n})\bigr)\]
	Hence, by induction in $\mathcal{M}$ we get
	\[\mathcal{M}\models\forall x\text{ } \phi'_0(x,\val{t_0},\ldots,\val{t_n})\]
	which by $\UTB^-$ again gives us $(\mathcal{M}, Tr')\models \forall x T(\phi'_0(\num{x},t_0,\ldots,t_k))$.
	Hence also 
	\[(\mathcal{M}, Tr)\models \forall x\text{ } T(\phi(\num{x},t_0,\ldots, t_k)),\]
	which ends the proof.
\end{proof}

In order to find a theory satisfying the Fischer-Horsten criterion, we decided to switch the inner logic of the truth theory. It allowed to formulate a very natural theory of truth modelled after Weak Kleene Scheme. Is it possible to realise Fischer-Horsten desiderata using a compositional theory of truth extending $\PT^-$? With the meaning we gave to the term ''axiomatic theory of truth'', we are not allowed to add more symbols to the language.\footnote{Without this restriction the answer is trivial: simply take $\PT^-$ together  with $(\WPT^- + \inter)$ but formulated with a different truth predicate symbol.} For the moment, we leave it as an open problem.

\bibliographystyle{plain}
\bibliography{MPTbib}


\end{document}